\documentclass[12pt]{amsart}
\usepackage{amscd,amsmath,amsthm,amssymb}
\usepackage{tikz}
\tikzstyle{punkt}=[circle, fill=black, minimum size=1mm,inner sep=0pt, draw]

\def\NZQ{\mathbb}               

\def\ZZ{{\NZQ Z}}

\def\FF{{\NZQ F}}
\def\GG{{\NZQ G}}

\def\TT{{\NZQ T}}

\def\G{{\mathcal G}}

\def\ab{{\mathbf a}}

\def\xb{{\mathbf x}}

\def\opn#1#2{\def#1{\operatorname{#2}}} 
\opn\chara{char} \opn\length{\ell} \opn\pd{pd} \opn\rk{rk}
\opn\projdim{proj\,dim} \opn\injdim{inj\,dim} \opn\rank{rank}
\opn\depth{depth} \opn\grade{grade} \opn\height{height}
\opn\embdim{emb\,dim} \opn\codim{codim}

\opn\Tr{Tr} \opn\bigrank{big\,rank}
\opn\superheight{superheight}\opn\lcm{lcm}
\opn\trdeg{tr\,deg}
\opn\reg{reg} \opn\lreg{lreg} \opn\ini{in} \opn\lpd{lpd}
\opn\size{size} \opn\sdepth{sdepth}
\opn\link{link}\opn\fdepth{fdepth}\opn\lex{lex}
\opn\tr{tr}
\opn\type{type}
\opn\gap{gap}
\opn\arithdeg{arith-deg}
\opn\astab{astab}
\opn\dstab{dstab}
\opn\bigheight{bigheight}
\opn\div{div} \opn\Div{Div} \opn\cl{cl} \opn\Cl{Cl}
\opn\Spec{Spec} \opn\Supp{Supp} \opn\supp{supp} \opn\Sing{Sing}
\opn\Ass{Ass} \opn\Min{Min}\opn\Mon{Mon}
\opn\Ann{Ann} \opn\Rad{Rad} \opn\Soc{Soc}
\opn\Im{Im} \opn\Ker{Ker} \opn\Coker{Coker} \opn\Am{Am}
\opn\Hom{Hom} \opn\Tor{Tor} \opn\Ext{Ext} \opn\End{End}
\opn\Aut{Aut} \opn\id{id}

\opn\nat{nat}
\opn\pff{pf}
\opn\Pf{Pf} \opn\GL{GL} \opn\SL{SL} \opn\mod{mod} \opn\ord{ord}
\opn\Gin{Gin} \opn\Hilb{Hilb}\opn\sort{sort}
\opn\PF{PF}\opn\Ap{Ap}
\opn\mult{mult}
\opn\aff{aff}
\opn\relint{relint} \opn\st{st}
\opn\lk{lk} \opn\cn{cn} \opn\core{core} \opn\vol{vol}  \opn\inp{inp} \opn\nilpot{nilpot}
\opn\link{link} \opn\star{star}\opn\lex{lex}\opn\set{set}
\opn\width{wd}
\opn\Fr{F}
\opn\QF{QF}
\opn\G{G}
\opn\type{type}\opn\res{res}
\opn\conv{conv}
\opn\gr{gr}

\def\pot#1#2{#1[\kern-0.28ex[#2]\kern-0.28ex]}

\opn\dirlim{\underrightarrow{\lim}}
\opn\inivlim{\underleftarrow{\lim}}
\let\sect=\cap
\let\dirsum=\oplus
\let\tensor=\otimes

\let\Dirsum=\bigoplus

\let\to=\rightarrow

\def\Implies{\ifmmode\Longrightarrow \else
	\unskip${}\Longrightarrow{}$\ignorespaces\fi}
\def\implies{\ifmmode\Rightarrow \else
	\unskip${}\Rightarrow{}$\ignorespaces\fi}
\def\iff{\ifmmode\Longleftrightarrow \else
	\unskip${}\Longleftrightarrow{}$\ignorespaces\fi}

\let\:=\colon
\newtheorem{Theorem}{Theorem}[section]
\newtheorem{Lemma}[Theorem]{Lemma}
\newtheorem{Corollary}[Theorem]{Corollary}
\newtheorem{Proposition}[Theorem]{Proposition}
\newtheorem{Remark}[Theorem]{Remark}

\newtheorem{Question}[Theorem]{Question}
\let\epsilon\varepsilon
\let\kappa=\varkappa
\def\qed{\ifhmode\textqed\fi
	\ifmmode\ifinner\quad\qedsymbol\else\dispqed\fi\fi}
\def\textqed{\unskip\nobreak\penalty50
	\hskip2em\hbox{}\nobreak\hfil\qedsymbol
	\parfillskip=0pt \finalhyphendemerits=0}
\def\dispqed{\rlap{\qquad\qedsymbol}}

\opn\dis{dis}
\def\pnt{{\raise0.5mm\hbox{\large\bf.}}}

\opn\Lex{Lex}

\begin{document}
	
	\title {Matchings and squarefree powers of edge ideals}

	\author {Nursel Erey, J\"urgen Herzog, Takayuki Hibi  and Sara Saeedi Madani}

	\address{Nusel Erey, Gebze Technical University, Department of Mathematics, 41400 Kocaeli, Turkey}
	\email{nurselerey@gtu.edu.tr}

	\address{J\"urgen Herzog, Fachbereich Mathematik, Universit\"at Duisburg-Essen, Campus Essen, 45117
		Essen, Germany}
	\email{juergen.herzog@uni-essen.de}

	\address{Takayuki Hibi, Department of Pure and Applied Mathematics, Graduate School of Information Science and Technology, Osaka University, Suita, Osaka 565--0871, Japan}
	\email{hibi@math.sci.osaka-u.ac.jp}

	\address{Sara Saeedi Madani, Faculty of Mathematics and Computer Science, Amirkabir University of Technology (Tehran Polytechnic), Tehran, Iran, and School of Mathematics, Institute for Research in Fundamental Sciences (IPM), Tehran, Iran}
	\email{sarasaeedi@aut.ac.ir}

	\dedicatory{ }
	
	\begin{abstract}
		Squarefree  powers of edge ideals are intimately related to matchings of the underlying graph. In this paper we give bounds for the regularity of squarefree powers of edge ideals, and we consider the question of when such powers are linearly related or have linear resolution. We also consider the so-called squarefree Ratliff property. 
	\end{abstract}
	
	\thanks{}
	
	\subjclass[2010]{05E40, 13D02, 05C70}
	
	
	
	\keywords{Matchings, edge ideals, squarefree powers}
	
	\maketitle
	
	\setcounter{tocdepth}{1}
	
	\section*{Introduction}
	
	The study of regularity of edge ideals arising from finite simple graphs and of their powers  is one of the current trends of commutative algebra.  In fact, in the last decade a huge number of papers on this topic  has been published.  In the present paper, instead of the ordinary powers of edge ideals, their {\em squarefree powers} will be systematically discussed.  Our study on squarefree powers of edge ideals is closely connected with the classical theory of matchings of finite simple graphs. This is part of the motivation to study squarefree powers.
	
	Let $G$ be a finite simple graph on $[n] = \{1, \ldots, n\}$ with $E(G)$ its edge set.  Recall that a finite graph $G$ is {\em simple} if $G$ admits no loop and no multiple edge.  Let $K$ be a field and $S = K[x_1, \ldots, x_n]$ the polynomial ring in $n$ variables over $K$.  We associate each edge $e = \{i, j\}$ of $G$ with the monomial $x_ix_j$ of $S$. We usually identify the edge $e$ and its corresponding monomial in $S$. The {\em edge ideal} of $G$ is the monomial ideal $I(G)$ of $S$ which is generated  by those monomials $e=x_ix_j$ with $e \in E(G)$.
	
	It was shown by Fr\"oberg \cite{F} that $I(G)$ has linear resolution if and only if the complementary graph $\overline{G}$ of $G$ is a chordal graph, where $\overline{G}$ is the graph with vertex set $[n]$ and edge set $\{ \{ i, j \} : \{ i, j \} \not\in E(G)\}$. Recall that a chordal graph is a finite simple graph in which any cycle of length greater than~4 has a chord. Furthermore, by virtue of Dirac's theorem on chordal graphs together with the modern theory of Gr\"obner bases, it was proved in \cite{HHZ2} that all powers of $I(G)$ have linear resolution if and only if $\overline{G}$ is chordal.
	
	One of the recent targets is to classify those graphs $G$ for which the second power $I(G)^2$ of $I(G)$ has linear resolution.  A finite simple graph $G$ is called {\em gap-free} if, for edges $e$ and $e'$ of $G$ with $e \cap e' = \emptyset$, there is $e'' \in E(G)$ with $e \cap e'' \neq \emptyset$ and $e' \cap e'' \neq \emptyset$.  It follows that if $I(G)^2$ has linear resolution, then $G$ is gap-free.  Its converse turned out to be false (Nevo and Peeva \cite{NP}).  On the other hand, several sufficient conditions are known (\cite{B, E1, E2}).  Nowadays, one of the reasonable conjectures is that $I(G)^k$ has linear resolution for all  $k \gg 0$ if and only if $G$ is gap-free.
	
	Now, apart from the subject of (ordinary) powers, we turn to the description of the present paper.  
	Let $I$ be a squarefree monomial ideal in $S$. For any $k\geq 1$ we denote by the $I^{[k]}$ the $k$th \emph{squarefree power} of $I$ which is defined to be the squarefree monomial ideal generated by the squarefree monomials in $I^k$. Note that $I^{[k]}=(0)$ for
	$k\gg 0$. Indeed, if we consider $I$ to be the edge ideal of a hypergraph, then the highest non-vanishing squarefree power of $I$ coincides with the so-called packing number of the hypergraph.
	
	In this paper we are interested in squarefree powers of edge ideals of graphs. A finite set of edges $M = \{e_1, \ldots, e_k\}$ of $G$ is called a {\em matching} of $G$ if $e_i \cap e_j = \emptyset$ for $1 \leq i < j \leq n$.  If a matching $M$ consists of $k$ edges, then $M$ is called a \emph{$k$-matching}.  The {\em matching number} of $G$, denoted by $\nu(G)$, is the maximum cardinality of matchings of $G$.  Given a matching $M = \{e_1, \ldots, e_k\}$, we write $u_M$ for the squarefree monomial $\prod_{i=1}^{k}e_i$ of $S$.  The squarefree $k$th power of the edge ideal $I(G)$ is 
	\[
	I(G)^{[k]} = ( u_M : \text{$M$ is a $k$-matching of $G$} )
	\]
	and $I(G)^{[k]} =(0)$ for $k > \nu(G)$.  Thus, in particular, $I(G)$ itself is the squarefree first power of $I(G)$. In \cite[Theorem~5.1]{BHZ} it was shown that  the last non-vanishing squarefree power of $I(G)$, namely $I(G)^{[\nu(G)]}$,  has linear quotients. This shows that squarefree powers of ideals have a quite different behaviour than ordinary powers. Another such instance is the behaviour of the function $g(k)=\mu(I(G)^{[k]})$.  Here $\mu(I)$ denoted the minimal number of generators of an ideal.  While the function $f(k)=\mu(I(G)^k)$  is a strictly increasing function, unless $G$ consists of only one edge, the function $g(k)$ in the range $1,\ldots,\nu(G)$ is unimodal, see \cite{HL1}.

	Following \cite{BHZ}, in the present paper, we found the basis for the study of squarefree powers of edge ideals.
	
	We sketch out this paper.  First, the role of Section $1$ is to fix notation and to supply fundamental results which will be required in the sequel.  Especially, in addition to the matching number $\nu(G)$ of a finite simple graph $G$, the {\em induced matching number} $\nu_1(G)$ together with the {\em restricted matching number} $\nu_0(G)$ of $G$ are introduced.  These three invariants of finite simple graphs play important roles in the study of squarefree powers of edge ideals.  
	
	Section $2$ is devoted to the study on bounds for regularity of squarefree powers of edge ideals.  It is known (\cite{BBH, BHT}) that
	\[
	2k + \nu_1(G) - 1 \leq \reg(I(G)^k) \leq 2k + \nu(G) - 1
	\]
	for all $k \geq 1$.  One cannot escape from the temptation to find the squarefree analogue of the above inequality.  In this paper, we prove that
	\[
	\reg(I(G)^{[k]}) \geq 2k + \nu_1(G) - k
	\]
	for all $1 \leq k \leq \nu_1(G)$.  On the other hand, in consideration of the fact that $I(G)^{[\nu(G)]}$ has linear quotients, it is reasonable to expect the inequalities
	\[
	\reg(I(G)^{[k]}) \leq 2k + \nu(G) - k
	\]
	for all $1 \leq k \leq \nu(G)$.  So far, the proposed upper bound is unknown except for $k = 1$ and $k = \nu(G)$.  We also prove this upper bound  for $k = 2$.
	
	The topic of Section $3$ is linearly related squarefree powers of edge ideals.  It is shown that if $I(G)^{[k]}$ is linearly related, then $I(G)^{[k+1]}$ is linearly related.  It then follows that there exists a smallest integer $\lambda(I(G))$ for which $I(G)^{[k]}$ is linearly related for all $k \geq \lambda(I(G))$.  It is known \cite[Lemma 5.2]{BHZ} that $\lambda(I(G)) \geq \nu_0(G)$.  In general, however, $\lambda(I(G)) > \nu_0(G)$ happens.  On the other hand, we prove that if $\nu_0(G) \leq 2$, then $I(G)^{[\nu_0(G)]}$ is linearly related.
	
	 Section $4$ is devoted to  the study  of  squarefree powers with linear resolution.  The highlight is the result that if a tree $G$ possesses a perfect matching, then $I(G)^{[\nu_0(G)]}$ has linear resolution. The proof of this result relies on the distinguished properties of perfect matchings. A {\em perfect matching} of a finite simple graph $G$ is a maximum matching $M$ for which each vertex of $G$ belong to an edge $e \in M$.  Perfect matchings play a leading role in the classical matching theory of graphs.
	
	Next, in Section $5$, the forests $G$ with $\nu_0(G) \leq 2$ are completely classified. This classification indeed provides the complete list of forests $G$ for which $I(G)^{[2]}$ has linear resolution. It would, of course, be of interest to classify all finite simple graphs $G$ with $\nu_0(G) \leq 2$.
	
	Finally, the target of Section $6$ is the squarefree Ratliff property.  In homage to Ratliff \cite{R}, we say that an ideal $I$ satisfies the {\em Ratliff property} if $I^{k+1} : I = I^k$ for all $k \geq 1$. Not all monomial ideals satisfy the Ratliff property. However it is known \cite{MMV} that every edge ideal satisfies the Ratliff property.  The corresponding colon ideals  of squarefree powers behave quite different.  In fact, it is shown that if $I$ is  any  squarefree monomial ideal, then $I^{[k]} : I = I^{[k]}$ for all $k \geq 2$.   When $I(G)$ is an edge ideal of $G$ with no isolated vertex, by using a simple technique of matchings, it can  be easily proved that $I(G)^{[k]} : I(G)^{[2]} = I(G)^{[k]}$ for $2 < k \leq \nu(G)$.  The question arises whether $I(G)^{[k]} : I(G)^{[\ell]} = I(G)^{[k]}$ for all $k$ and all $\ell<k$. We conclude the present paper by giving a partial answer to this question.  To this end  we discuss edge ideals of equimatchable graphs.  A finite simple graph $G$ is called {\em equimatchable} if every maximal matching of $G$ is maximum.  It is shown that if $G$ is equimatchable and if $k$ and $\ell$ are integers with $1 \leq \ell < k \leq \nu(G)$, then $I(G)^{[k]} : I(G)^{[\ell]} = I(G)^{[k]}$.

	\section{Preliminaries}
	
		In this section we fix notation and recall some technical results which will be used in the sequel. Throughout this paper $S$ denotes the polynomial ring $K[x_1\dots, x_n]$ in $n$ indeterminates, where  $K$ is a field. Let $I$ be a  monomial ideal in $S$.  We denote by $G(I)$ the unique set of monomial generators of $I$.
	
	
	Let graph $G$ be a finite simple graph on the vertex set $V(G)=\{x_1,\ldots, x_n\}$ and the edge set $E(G)$. 
	In this paper we often identify an edge $e=\{a,b\}$ with its corresponding monomial $e=ab$ and we denote the edge by $(ab)$. Throughout the paper we denote by $G-e$ the graph on the same vertex set as $G$ obtained by removing the edge $e$ from $G$. For any $W\subseteq V(G)$, the induced subgraph of $G$ on $W$, denoted by $G_W$ is the graph with the vertex set $W$ and the edge set $\{e\in E(G): e\subseteq W\}$.  Then, for any $W\subseteq V(G)$ we denote by $G-W$ the induced subgraph $G_{V(G)\setminus W}$ of $G$. We denote the set of vertices of $G$ which are adjacent to a vertex $x$ of $G$ by $N_G(x)$. Moreover, we set $N_G[x]=N_G(x)\cup \{x\}$. 
	
	
	An \emph{induced matching} of $G$ is a matching which is an induced subgraph of $G$. We denote by $\nu_1(G)$ the \emph{induced matching number} of $G$ which is the maximum size of an induced matching in $G$. 
	An induced matching of size~2 is called a \emph{gap} of $G$. 
	A \emph{restricted matching} of $G$ was defined in \cite{BHZ} to be a matching in which there exists an edge which provides a gap in $G$ with any other edge in the matching. The maximum size of a restricted matching in $G$ is denoted by $\nu_0(G)$. It follows from the definitions that
	\[
	\nu_1(G)\leq \nu_0(G)\leq \nu(G).
	\]

	Let $I$ be a monomial ideal in $S$ generated in degree $d$. We define the graph $G_I$ whose vertex set is $G(I)$ and its edge set is
	\[
	E(G_I)=\{\{u,v\}: u,v\in G(I)~\text{with}~\deg (\lcm(u,v))=d+1 \}.
	\]

	For all $u, v\in G(I)$ let  $G^{(u,v)}_I$ be the induced subgraph of $G_I$ with vertex set
	\[
	V(G^{(u,v)}_I)=\{w\in G(I)\: \text{$w$ divides $\lcm(u, v)$}\}.
	\]
	
	\medskip
	Let $I\subset S$ be a graded ideal generated in a single degree. The ideal $I$ is said to be {\em linearly related}, if the first syzygy module of $I$ is generated by linear relations. 
	
	\medskip
	The following result was shown in \cite[Corollary~2.2]{BHZ}.
	
	\begin{Theorem} \label{connected}
		Let $I$ be a monomial ideal generated in degree $d$. Then $I$ is linearly related if and only if for all $u,v\in G(I)$  there is a path in $G^{(u,v)}_I$ connecting $u$ and $v$.	
	\end{Theorem}
	
	For the proof of the next corollary  we will use the so-called restriction lemma.

	\begin{Lemma}
		{\em (\cite[Lemma 4.4]{HHZ})}
		\label{restrict}
		Let $I\subset S$ be a monomial ideal, and let $\FF$ be its minimal multigraded free $S$-resolution. Furthermore, let  $m$ be a monomial. We set
		\[
		I^{\leq m}=(u\in   G(I)\:\;  u|m ).
		\]
		Let $F_i=\Dirsum_{j}S(-\ab_{ij})$. Then $\FF^{\leq m}$ with $$F_i^{\leq m}= \Dirsum_{j,\;  \xb^{\ab_{ij}}|m}S(-\ab_{ij})$$ is a subcomplex of $\FF$ and the minimal multigraded free resolution of $I^{\leq m}$.
	\end{Lemma}
	
	As an immediate consequence of Lemma~\ref{restrict} we get
	
	\begin{Corollary}
		\label{restriction}
		Let $H$ be an induced subgraph of $G$. Then $b_{i,\ab}(I(H)^{[k]})\leq b_{i,\ab}(I(G)^{[k]})$ for all $i$ and $\ab \in \ZZ^n$. In particular,  $\projdim(I(H)^{[k]})\leq  \projdim(I(G)^{[k]})$ and $\reg(I(H)^{[k]})\leq  \reg(I(G)^{[k]})$.
	\end{Corollary}
	
	\begin{Corollary}\label{connected-one direction}
		Let $I$ be a monomial ideal generated in degree $d$, and let $m$ be a monomial of degree $d+2$. Suppose for any $u,v\in G(I)$ with $\mathrm{lcm}(u,v)=m$ there exists a path in $G_{I}^{(u,v)}$ connecting $u$ and $v$. Then $b_{1,m}(I)=0$.
	\end{Corollary}
	
	\begin{proof}
		Let $I'=I^{\leq m}$. By the restriction lemma we have $b_{1,w}(I)=b_{1,w}(I')$ for all $w|m$. Therefore, it suffices to show that $I'$ is linearly related. Let $u,v\in G(I')$. If $\deg(\lcm(u,v))=d+2$,  then $\lcm(u,v)=m$. Therefore, by our assumption there exists a path in $G_{I}^{(u,v)}=G_{I'}^{(u,v)}$ connecting $u$ and $v$. If $\deg(\lcm(u,v))=d+1$, then $u$ and $v$ are connected by an edge by definition. Hence Theorem~\ref{connected} implies that $I'$ is linearly related.
	\end{proof}

	The next lemma describes another case for the vanishing of $b_{1,m}(I)$.
	
	\begin{Lemma}
		\label{taylor}
		Let $I$ be a monomial ideal and let $m$ be a monomial. Suppose for all $u,v\in G(I)$ with $\lcm(u,v)=m$ there exists $w\in G(I)\setminus \{u,v\}$ such that
		\begin{enumerate}
			\item[(i)] $\lcm(u,v,w)=m$;
			\item[(ii)] $\lcm(u,w)\neq m$ and $\lcm(v,w)\neq m$.
		\end{enumerate}
		Then $b_{1,m}(I)=0.$
	\end{Lemma}

	\begin{proof}
		Let $G(I)=\{u_1,\ldots,u_r\}$. For the proof we use the Taylor resolution $\TT$ associated with $u_1,\ldots,  u_r$ (see~\cite{T}). Recall that $T_k$ is a free $S$-module with basis $e_{i_1}\wedge e_{i_2}\wedge \cdots \wedge e_{i_k}$ with $1\leq i_1<i_2<\cdots <i_k\leq r$ and
		\[
		\deg(e_{i_1}\wedge e_{i_2}\wedge \cdots \wedge e_{i_k})=\lcm(u_{i_1},u_{i_2},\ldots,u_{i_k}).
		\]
		The  chain map $T_2\to T_1$ is given by $e_i\wedge e_j\mapsto (\lcm(u_i,u_j)/u_j)e_j-(\lcm(u_i,u_j)/u_i)e_i$, and $T_3\to T_2$ is defined by
		\[
		e_i\wedge e_j\wedge e_k\mapsto \frac{\lcm(u_i,u_j,u_k)}{\lcm(u_j,u_k)}e_j\wedge e_k-\frac{\lcm(u_i,u_j,u_k)}{\lcm(u_i,u_k)}e_i\wedge e_k +\frac{\lcm(u_i,u_j,u_k)}{\lcm(u_i,u_j)}e_i\wedge e_j.
		\]
		Since the Taylor resolution is a $S$-free resolution of $S/I$ it follows that
		\begin{eqnarray}
		\nonumber 
		\label{beta}
		b_{1,m}(I)=\dim_K H_2(\TT\tensor_SK)_m.
		\end{eqnarray}
		
		The cycles in $\TT\tensor_SK$ of degree $m$ and of homological degree $2$ are all of the form $(e_i\wedge e_j)\tensor 1$ with $i<j$ and $\lcm(u_i,u_j)=m$. Thus we have to show that these elements are boundaries in $\TT\tensor_SK$. We set $u=u_i$ and  $v=u_j$. By our assumptions there exists $w\in G(I)\setminus \{u,v\}$ satisfying (i) and (ii). There exists $k\neq i,j$ such that $w=u_k$. Without loss of generalities we may assume that $i<j<k$. Then $(e_i\wedge e_j\wedge e_k)\tensor 1\mapsto (e_i\wedge e_j)\tensor 1$, because of (i) and (ii).
	\end{proof}

	\section{Bounds for the regularity of squarefree powers of edge ideals}
	
	In this section we give upper and lower bounds of the regularity of squarefree powers of edge ideals.
	\begin{Theorem}
		\label{lower}
		Let $G$ be graph.  Then $\reg(I(G)^{[k]})\geq 2k+\nu_1(G)-k$ for all $k\leq \nu_1(G)$.
	\end{Theorem}
	
	\begin{proof}
		Let $r=\nu_1(G)$. There exists an induced subgraph $H$ with $r$ edges which forms an induced matching of $G$. By Corollary~\ref{restriction}, it follows that $\reg(I(G)^{[k]})\geq \reg(I(H)^{[k]})$. Thus it suffices to show that $\reg(I(H)^{[k]})\geq 2k+r-k$. We set $I=I(H)$ and claim that
		$$b_{r-k,2r}(I^{[k]})\neq 0.$$
		The theorem follows once we have proved this claim.
		
		We may assume that $I=(x_1y_1,\ldots, x_ry_r)$. Let $J=(z_1,\ldots,z_r)$, where $\{z_1,\ldots,z_r\}$ is a new set of variables.  Then $J^{[k]}$ is a squarefree strongly stable ideal in the polynomial ring $R=K[z_1,\ldots,z_r]$.
		
		We first show that $b_{r-k,r}(J^{[k]})\neq 0$. By \cite[Theorem 7.4.1]{HH}, $b_{r-k,r}(J^{[k]})\neq 0$ if there exists a subset $F\subset [r]$ and $u\in G(J^{[k]})$ such that (i) $|F|=r-k$, (ii)~$\max\{i\:\; i\in F\}< \max\{i\: z_i|u\}$ and  (iii) $F\sect \supp(u)=\emptyset$.  These conditions can be realized by choosing $F=\{1,2,\ldots,r-k\}$ and $u=z_{r-k+1}\cdots z_r$.
		
		Consider the map $\varphi\: R\to S=K[x_1,\ldots,x_r, y_1,\ldots,y_r]$, $z_i\mapsto x_iy_i$ for $i=1,\ldots, r$, and let $\FF$ be the graded  minimal free resolution of $J^{[k]}$ over $R$. Since $x_1y_1,\ldots, x_ry_r$ is a regular sequence on $S$, the $K$-algebra homomorphism $\varphi$ is flat, and it follows that $\GG:= \FF\tensor_RS$ is the minimal free resolution  of $I^{[k]}$ over $S$, and the  entries of the chain maps of $\GG$ are obtained by applying $\varphi$ to the entries of the chain maps of $\FF$.  Since $J^{[k]}$ has $k$-linear resolution, it follows that $G_i=S(-2k-2i)^{b_i(I^{[k]})}$ for all $i$. Thus, since $b_{i,j}(J^{[k]})=b_{i,2j}(I^{[k]})$ for any $i,j$, and since  $b_{r-k,r}(J^{[k]})\neq 0$,  the desired result follows.
	\end{proof}

	For the regularity of ordinary  powers of edge ideals the following upper bound was established.
	
	\begin{Theorem}{\em (\cite[Theorem 3.4]{BBH})}
		\label{established}
		For all $k\geq 1$, $\reg(I(G)^k)\leq 2k+\nu(G)-1$.
	\end{Theorem}

	By applying the restriction lemma (Lemma~\ref{restrict}) we see that any upper bound for the regularity of ordinary  powers of edge ideals is also an upper bound for the regularity of the squarefree powers.
	
	We expect however a stronger   bound for squarefree powers.
	
	\begin{Question}\label{upper}
		For all $k\leq \nu(G)$ does $\reg(I(G)^{[k]})\leq 2k+\nu(G)-k$ hold?
	\end{Question}
	
	This upper bound is valid for $k=1$ because of Theorem~\ref{established}, and for $k=\nu(G)$  because of the following theorem
	
	\begin{Theorem} \label{highest power}
		{\em(}\cite[Theorem~5.1]{BHZ}{\em)}
		Let $G$ be a graph. Then $I(G)^{[\nu(G)]}$ has linear quotients.
	\end{Theorem}
	
	Recall that a monomial ideal has \emph{linear quotients} if its minimal generators can be ordered as $u_1\ldots,u_m$ such that the ideal
	$(u_1,\ldots,u_{j-1}):u_j$ is generated by variables for each $j=2,\ldots,m$. It is well known that if a monomial ideal generated in same degree has linear quotients, then it has linear resolution as well.

\medskip
Now we prove the proposed upper bound in Question~\ref{upper} for $k=2$. To prove this result, the following lemmata are required. For the rest of this section we adopt the convention that 
$\reg (I)=1$ when $I=(0)$.   

\begin{Lemma}\label{majaglassbilye}
	Let $(ab)$ be an edge of $G$. Then
	\[I(G)^{[2]}:ab =I(G-\{a,b\})+(cd : c\in N_G(a),\, d\in N_G(b),\, c\neq d \text{ and } c,d\notin \{a,b\}).\]
	In particular, $I(G)^{[2]}:ab$ is the edge ideal of a graph.  
\end{Lemma}
\begin{proof}
	It is clear that $I(G)^{[2]}:ab$ contains the ideal on the right-hand side. To see the other containment, 
	let $u$ be a monomial in $I(G)^{[2]}:ab$. Then there exist variables $x,y$ different from $a,b$ with $x\neq y$ dividing $u$ such that either $(xy)(ab)$ or $(xa)(yb)$ is a $2$-matching of $G$. This implies that $u$ belong to the right-hand side. 
\end{proof}

The following lemma is well known, see for example \cite[Lemma~2.10]{DHS}. 

\begin{Lemma}\label{afternoon}
	Let $I$ be a monomial ideal and let $u$ be a monomial of degree $d$. Then
	\begin{itemize}
		\item[(i)] $\reg(I)\leq \max\{\reg(I:u)+d, \reg(I,u)\}$,
		\item[(ii)] for any variable $x$, $\reg(I,x)\leq \reg(I).$
	\end{itemize}
\end{Lemma}


\begin{Lemma}\label{disjoint}
	Let $I$ and $J$ be monomial ideals generated in disjoint sets of variables. Then $\reg(I+J)=\reg(I)+\reg(J)-1$.
\end{Lemma}

Lemma~\ref{disjoint} is a simple consequence of the fact that the minimal graded free resolution of $S/(I+J)$ is isomorphic to the tensor product of the minimal graded free resolutions of $S/I$ and $S/J$.  

\begin{Lemma}\label{separate case}
	Let $(ab)$ be an edge of $G$. Suppose that $N_G(x)\subseteq \{a,b\}$ for each $x\in (N_G(a)\cup N_G(b))\setminus \{a,b\}.$
	Then $\reg (I(G)^{[2]}:ab)\leq \nu(G)$.
\end{Lemma}
\begin{proof}
	Let $H$ be the graph whose edge ideal is $I(G)^{[2]}:ab$.
	By our assumption, the induced subgraph of $G$ on $N_G(a)\cup N_G(b)$ is a connected component of $G$. Therefore, $G$ is the disjoint union of $G_{N_G(a)\cup N_G(b)}$ and some subgraph, say $G_1$. Observe that $\nu(G_{N_G(a)\cup N_G(b)})\leq2$ and $\nu(G)=\nu(G_1)+\nu(G_{N_G(a)\cup N_G(b)})$. Let $H_1$ be the graph defined by $I(H_1):=(xy : x\in N_G(a),\, y\in N_G(b),\, x\neq y \text{ and } x,y\notin \{a,b\})$. By Lemma~\ref{majaglassbilye}, we get 
	\[I(H)=I(G_1)+I(H_1)\]
	where the ideals $I(G_1)$ and $I(H_1)$ are generated in disjoint sets of variables. Therefore, by Lemma~\ref{disjoint} we have
	\begin{equation}\label{eq:disjoint}
	\reg(I(H))=\reg(I(G_1))+\reg(I(H_1))-1. 
	\end{equation}
	Observe that $H_1^c$ is a disjoint union of some  complete graphs. Therefore by \cite[Theorem~1]{F} we have $\reg(I(H_1))\leq 2$. Using Eq.\eqref{eq:disjoint} and Theorem~\ref{established}, we get
	\[\reg(I(H))\leq \nu(G_1)+1 + \reg(I(H_1))-1 =\nu(G_1)+ \reg(I(H_1)).\]
	If $\nu(G_{N_G(a)\cup N_G(b)})=2$, then the result follows. Next, assume that $\nu(G_{N_G(a)\cup N_G(b)})=1$. This implies that $H_1$ has no edges. Indeed, if $(xy)$ is an edge of $H_1$ for some $x\in N_G(a)$ and $y\in N_G(b)$, then the edges $(xa)$ and $(yb)$ provide a 2-matching in $G_{N_G(a)\cup N_G(b)}$, a contradiction. Thus, by our convention we have $\reg(I(H_1))= 1$, which together with Eq.\eqref{eq:disjoint} and Theorem~\ref{established} yield the desired result.   
\end{proof}

In the next lemma, we show that the statement of Lemma~\ref{separate case} remains valid without any extra assumption on the choice of the edge $(ab)$.  

\begin{Lemma}\label{colon}
	Let $(ab)$ be an edge of $G$. Then $\reg(I(G)^{[2]}:ab)\leq \nu(G)$.
\end{Lemma}
\begin{proof}
	We proceed by induction on the number of vertices of the graph. If the graph $G$ is just the edge $(ab)$, then the result follows from our convention about the zero ideal. So, suppose that $G$ is not just an edge, and let $H$ be the graph such that $I(H)=I(G)^{[2]}:ab$. Now we distinguish the following cases: 
	
	\medskip
	Case~1:
	Suppose that $N_G(a)\cup N_G(b)=\{a,b\}$. Then $H=G-\{a,b\}$ and $\nu(G)= \nu(H)+1$. Therefore, by Theorem~\ref{established}, we obtain $\reg(I(H))\leq \nu(H)+1=\nu(G)$.
	
	\medskip 
	Case~2: Suppose that $N_G(a)\cup N_G(b)\neq\{a,b\}$. By Lemma~\ref{separate case}, we may assume that there exists a vertex $x \in (N_G(a)\cup N_G(b))\setminus \{a,b\}$ such that $N_G(x)\nsubseteq \{a,b\}$. We may assume that $(ax)$ is an edge of $G$, and let $y\in N_G(x)$ with $y\notin \{a,b\}$. By Lemma~\ref{afternoon}, we have
	\[\reg(I(H))\leq \max\{\reg(I(H-\{x\})),\, \reg(I(H-N_H[x]))+1 \}.\]
	Then, it remains to show that $\reg(I(H-\{x\}))\leq \nu(G)$ and $\reg(I(H-N_H[x]))\leq \nu(G)-1$. Observe that
	\[I(H-\{x\})=I(G-\{x\})^{[2]}:ab\]
	and thus by induction hypothesis
	\[\reg(I(H-\{x\}))\leq \nu(G-\{x\})\leq \nu(G).\]
	Notice that $H-N_H[x]=G-(N_G[b]\cup N_G[x])$ and any matching of $G-(N_G[b]\cup N_G[x])$ can be extended to a matching of $G$ by adding the edges $(ab)$ and $(xy)$. Therefore, $\nu(H-N_H[x])+2\leq \nu(G)$. Since $\reg(I(H-N_H[x]))\leq \nu(H-N_H[x])+1$ by Theorem~\ref{established}, the result follows.
\end{proof}

\begin{Lemma}\label{generated-by-variables}
	Let $G$ be a graph with $I(G)=(e_1,\ldots, e_r)$. Then for every $i>1$
	\[(I(G)^{[2]},e_1,\ldots ,e_{i-1}):e_i = (I(G)^{[2]}:e_i)+J_i\]
	where either $J_i=(0)$ or $J_i$ is an ideal generated by some variables.
\end{Lemma}
\begin{proof}
	For any $j<i$, if $e_j$ and
	$e_i$ have a common vertex, then $e_j:e_i$ is a variable generator of $J_i$. Otherwise $e_j$ is a generator of $I(G)^{[2]}:e_i$.
\end{proof}

Now we are ready to prove our proposed upper bound for the second squarefree power of edge ideals.  

\begin{Theorem}
	Let $G$ be a graph. Then $\reg(I(G)^{[2]})\leq 2+\nu(G)$.
\end{Theorem}
\begin{proof}
	Let $I(G)=(e_1,\ldots ,e_r)$. First note that by Lemma~\ref{afternoon} and Lemma~\ref{generated-by-variables}, for every $i>1$ we have
	\[\reg((I(G)^{[2]},e_1,\ldots ,e_{i-1}):e_i) \leq \reg((I(G)^{[2]}:e_i)). \]
	Keeping Lemma~\ref{colon} in mind, we apply Lemma~\ref{afternoon} repeatedly and we get
	\begin{eqnarray*}
		\reg(I(G)^{[2]})&\leq & \max\{\reg(I(G)^{[2]}:e_1)+2, \reg(I(G)^{[2]},e_1)\}\\
		&\leq & \max\{\nu(G)+2,  \reg(I(G)^{[2]},e_1)\}\\
		&\leq & \max\{\nu(G)+2,\reg((I(G)^{[2]},e_1):e_2)+2, \reg(I(G)^{[2]},e_1,e_2)\}\\
		&\leq & \max\{\nu(G)+2,\reg(I(G)^{[2]},e_1,e_2)\}\\
		&\leq & \vdots\\
		&\leq & \max\{\nu(G)+2, \reg(I(G))\}\\
		&\leq & \max\{\nu(G)+2, \nu(G)+1\}=\nu(G)+2,
	\end{eqnarray*}
	where the last inequality follows from Theorem~\ref{established}.
\end{proof}

	\section{Linearly related squarefree powers}
	
	Let $G$ be a graph on $n$ vertices. In this section we want to discuss which squarefree powers $I(G)^{[k]}$ of $I(G)$ are linearly related. Unlike the ordinary powers (see~\cite[Theorem~3.1]{BHZ}), being linearly related at some squarefree power does not guarantee the same property for all squarefree powers, by 
	\cite[Lemma~5.2]{BHZ} and Theorem~\ref{highest power}. However, in the following, we show that if some squarefree power is linearly related, then all higher squarefree powers are linearly related as well.   
	
	
	\begin{Theorem}
		\label{next}
		Suppose $I(G)^{[k]}$ is linearly related. Then $I(G)^{[k+1]}$ is linearly related.
	\end{Theorem}
	
	\begin{proof}
		Let $I=I(G)$, and let $u, v\in G(I^{[k+1]})$ with $u\neq v$. By Theorem~\ref{connected}, we need to show that there is a path between $u$ and $v$ in the graph $G_{I^{[k+1]}}^{(u,v)}$. 
		Let
		$u=(a_1b_1)\cdots (a_{k+1} b_{k+1})$ and
		$v=(a'_1b'_1)\cdots (a'_{k+1} b'_{k+1})$.
		
		We claim that there exist indices $i$ and $j$ such that
		\[
		\{a_i,b_i\}\cap \{a'_1,b'_1,\ldots,\hat{a'_j},\hat{b'_j},\ldots, a'_{k+1},b'_{k+1}\}=\emptyset.
		\]
		Assuming the claim is proved, we find a desired path between $u$ and $v$ as follows. We may assume that $i=j=1$. By our assumption $I^{[k]}$ is linearly related. Thus Theorem~\ref{connected} implies that there is a path connecting $u'=(a_2b_2)\cdots (a_{k+1} b_{k+1})$ and
		$v'=(a'_2b'_2)\cdots (a'_{k+1} b'_{k+1})$ in $G_{I^{[k]}}^{(u',v')}$, say $u',z_1,\ldots,z_t,v'$. Then it follows from the choice of $i$ and $j$ that the sequence $u=(a_1b_1)u',(a_1b_1)z_1,\ldots,(a_1b_1)z_t,(a_1b_1)v'$ is a path connecting $u$ and $(a_1b_1)v'$ in the graph $G_{I^{[k+1]}}^{(u,(a_1b_1)v')}$. We denote this path by $P$. Now, let $w=(a_1b_1)(a'_2b'_2)\cdots (a'_{k} b'_{k})$ and $w'=(a_1'b_1')(a'_2b'_2)\cdots (a'_{k} b'_{k})$. Since $I^{[k]}$ is linearly related, Theorem~\ref{connected} implies that there is a path $w,y_1,\ldots,y_s,w'$ connecting $w$ and $w'$ in $G_{I^{[k]}}^{(w,w')}$. Then  $w(a'_{k+1}b'_{k+1}),y_1(a'_{k+1}b'_{k+1}),\ldots,y_s(a'_{k+1}b'_{k+1}),
		w'(a'_{k+1}b'_{k+1})=v$ is a path connecting $w(a'_{k+1}b'_{k+1})$ and $v$ in $G_{I^{[k+1]}}^{(w(a'_{k+1}b'_{k+1}),v)}$. We denote this path by $P'$. Finally, since $w(a'_{k+1}b'_{k+1})=(a_1b_1)v'$, we may compose the paths $P$ and $P'$ in $G_{I^{[k+1]}}^{(u,v)}$ to connect $u$ and $v$.
		
		Proof of the claim: Since $u\neq v$, there exists an $i$ such that $a_i$ or $b_i$ does not belong to the support of $v$. We may assume that $a_i\notin \supp (v)$. If $b_i\notin \supp (v)$, then we may choose any $j=1,\ldots,k+1$. On the other hand, if
		$b_i\in \supp (v)$, namely $b_i\in \{a'_\ell,b'_\ell\}$ for some
		$\ell=1,\ldots, k+1$, then we may choose $j=\ell$.
	\end{proof}
	
	By Theorem~\ref{highest power} we know that the highest nonzero squarefree power of $I(G)$, namely  $I(G)^{[\nu(G)]}$, has linear quotients. Thus, due to Theorem~\ref{next}, there exists a smallest number $\lambda(I(G))$ with the property that  $I(G)^{[k]}$ is linearly related for all $k\geq \lambda(I(G))$. In \cite[Lemma~5.2]{BHZ} it was shown that $\lambda(I(G))\geq \nu_0(G)$.  In the same paper it was conjectured that $\lambda(I(G))= \nu_0(G)$, see \cite[Conjecture~5.3]{BHZ}. However,  the following examples show that in general  $\lambda(I(G)) \neq \nu_0(G)$. Here we refer to this conjecture as the  $\nu_0$-conjecture. 
		
	For example the graph $G$ shown in Figure~\ref{counterexample} has $\nu_0(G)=3$ while our computations with CoCoA shows that $I(G)^{[3]}$ is not linearly related as its Betti diagram is displayed below.
	
	\medskip 
	{
		\begin{verbatim}
	       0     1     2    
		---------------------
		6:    14    19     6    
		7:     -     1     1     
		---------------------
		Tot:  14    20     7   
		\end{verbatim}
	}

\medskip 	
	Also, for the graph $G'$ shown in Figure~\ref{two squares} we have $\nu_0(G')=3$ and $I(G')^{[3]}$ is not linearly related as we see in its Betti diagram in the following, even though the edge ideal of each component of $G'$ and its second squarefree power have both linear resolution.

	\medskip 
	
	{
		\begin{verbatim}
		       0    1     2    
		--------------------
		6:     8    8     2    
		7:     -    1     -     
		---------------------
		Tot:  8    9     2   
		\end{verbatim}
	}  

\medskip 

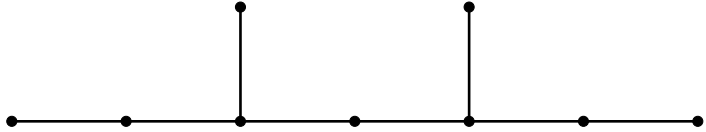
\begin{figure}[h!]
	\centering 
	\begin{tikzpicture}[scale = 0.76]
	\draw [line width=1pt] (1,7)-- (3,7);
	\draw [line width=1pt] (3,7)-- (5,7);
	\draw [line width=1pt] (5,7)-- (7,7);
	\draw [line width=1pt] (7,7)-- (9,7);
	\draw [line width=1pt] (9,7)-- (11,7);
	\draw [line width=1pt] (11,7)-- (13,7);
	\draw [line width=1pt] (5,7)-- (5,9);
	\draw [line width=1pt] (9,7)-- (9,9);
	\begin{scriptsize}
	\draw [fill=black] (1,7) circle (2.5pt);
	\draw [fill=black] (3,7) circle (2.5pt);
	\draw [fill=black] (5,7) circle (2.5pt);
	\draw [fill=black] (7,7) circle (2.5pt);
	\draw [fill=black] (9,7) circle (2.5pt);
	\draw [fill=black] (11,7) circle (2.5pt);
	\draw [fill=black] (13,7) circle (2.5pt);
	\draw [fill=black] (5,9) circle (2.5pt);
	\draw [fill=black] (9,9) circle (2.5pt);
	\end{scriptsize}
	\end{tikzpicture}
	\caption{A connected counterexample to the  $\nu_0$-conjecture}
	\label{counterexample}
\end{figure}

\medskip

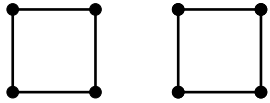
\begin{figure}[h!]
	\centering 
	\begin{tikzpicture}[scale=1.1]
	\draw [line width=1pt] (-7,3)-- (-6,3);
	\draw [line width=1pt] (-7,3)-- (-7,2);
	\draw [line width=1pt] (-6,3)-- (-6,2);
	\draw [line width=1pt] (-7,2)-- (-6,2);
	\draw [line width=1pt] (-4,3)-- (-4,2);
	\draw [line width=1pt] (-5,3)-- (-5,2);
	\draw [line width=1pt] (-5,3)-- (-4,3);
	\draw [line width=1pt] (-5,2)-- (-4,2);
	\begin{scriptsize}
	\draw [fill=black] (-7,3) circle (2pt);
	\draw [fill=black] (-6,3) circle (2pt);
	\draw [fill=black] (-6,2) circle (2pt);
	\draw [fill=black] (-7,2) circle (2pt);
	\draw [fill=black] (-4,3) circle (2pt);
	\draw [fill=black] (-4,2) circle (2pt);
	\draw [fill=black] (-5,3) circle (2pt);
	\draw [fill=black] (-5,2) circle (2pt);
	\draw [fill=black] (-5,3) circle (2pt);
	\draw [fill=black] (-4,3) circle (2pt);
	\draw [fill=black] (-5,2) circle (2pt);
	\draw [fill=black] (-4,2) circle (2pt);
	\end{scriptsize}
	\end{tikzpicture}
	\caption{A disconnected counterexample to the  $\nu_0$-conjecture}
	\label{two squares}
\end{figure}

	
	The next result shows that the $\nu_0$-conjecture holds when $\nu_0(G)=2$.
	
	\begin{Theorem}
		\label{true}
		Let $G$ be a graph with $\nu_0(G)\leq 2$. Then  $I(G)^{[k]}$ is linearly related for all $k\geq 2$.
	\end{Theorem}
	We postpone the proof of this theorem to the end of this section. First we would like to remark that in the statement of Theorem~\ref{true}, the property of being linearly related can not be replaced by having linear resolution. For example, let $G=C_7$ be the $7$-cycle. Then $\nu_0(C_7)=2$. Thus, by Theorem~\ref{true}, $I(C_7)^{[2]}$ is linearly related. However, $I(C_7)^{[2]}$ does not have linear resolution, as can be seen from its Betti diagram given by CoCoA. 

\medskip 	
	
	{
		\begin{verbatim}
		       0     1     2    
		---------------------
		4:    14    21     7    
		5:     -     -     1     
		---------------------
		Tot:  14    21     8   
		\end{verbatim}
	}

	\medskip 
	For the proof of Theorem~\ref{true}, we need the following general result.

	\begin{Theorem}
		\label{nurselsays}
		Let $G$ be a graph, and let $k\geq 2$. Then $b_{1,p}(I(G)^{[k]})=0$ for all $p\geq 3k+1$.
	\end{Theorem}
	
	\begin{proof}
		Let $I=I(G)$, and let $u_1,u_2\in G(I^{[k]})$  with $\lcm(u_1,u_2)=u$ and $\deg (u)\geq 3k+1$. By Lemma~\ref{taylor} it is enough to show that there exists $u_3\in G(I^{[k]})$ such that the conditions (i) and (ii) of Lemma~\ref{taylor} are satisfied.
		
		Let $u_1=(x_1y_1)(x_2y_2)\cdots (x_ky_k)$ where each $(x_iy_i)\in I$.  Since $\deg (u)\geq 3k+1$,  there are at least $k+1$ variables which divide $u_2$ but do not  divide $u_1$. Suppose that $u_2= e_1e_2\cdots e_k$ for some $e_i\in G(I)$.  Then there exists $i$ such that $e_i=x_{k+1}y_{k+1}$ and $\{x_{k+1}, y_{k+1}\}\sect \supp(u_1)=\emptyset$. For simplicity we may assume that $i=1$.
		
		We also conclude that $|\supp(u_1)\sect\supp(u_2)|\leq k-1$. Suppose that $(\supp(u_1)\sect\supp(u_2))\sect \{x_j,y_j\}\neq \emptyset$ for all $j=1,\ldots,k$. Then $|\supp(u_1)\sect\supp(u_2)|\geq k$, a contradiction. Therefore, there exists
		$x_jy_j$ for some $1\leq j\leq k$ such that $(\supp(u_1)\sect\supp(u_2))\sect \{x_j,y_j\}=\emptyset$. Hence,   $\supp(u_2)\sect \{x_{j}, y_{j}\}=\emptyset$, because $\{x_{j}, y_{j}\}\subset \supp(u_1)$.
		
		We may assume that $j=1$. We choose $u_3= (x_2y_2)(x_3y_3)\ldots (x_ky_k)(x_{k+1}y_{k+1})$ in $G(I^{[k]})$. Then for this $u_3$, (i) is obviously satisfied. Moreover, condition (ii) holds, because $x_1\nmid \lcm(u_2,u_3)$ and $\lcm(u_1,u_3)$ has degree $2k+2$.
	\end{proof}
	We are now ready to prove Theorem~\ref{true}.
	\begin{proof}
		[Proof of Theorem~\ref{true}]
		Because of Theorem~\ref{next}, it suffices to show that $I^{[2]}=I(G)^{[2]}$ is linearly related. By using the Taylor resolution \cite{T} associated to $I^{[2]}$, we see that $b_{1,p}(I^{[2]})=0$ for all $p>8$. By Theorem~\ref{nurselsays}, we also have $b_{1,p}(I^{[2]})=0$ for $p=7,8$. Hence it remains to show that $b_{1,m}(I^{[2]})=0$ for any squarefree monomial $m$ of degree $6$. By Corollary~\ref{connected-one direction}, it suffices to show that for any $u,v\in G(I^{[2]})$ with $\lcm (u,v)=m$, there is a path in $G_{I^{[2]}}^{(u,v)}$ connecting $u$ and $v$.
		
		\medskip
		Let $m=abcdef$ and $u=(ab)(cd)$. It is enough  to consider the following cases:
		
		\medskip
		Case~1.  $v=(ab)(ef)$.  Since $\nu_0(G)\leq 2$, the edge $(ab)$ does not form a gap with $(cd)$ and $(ef)$. Therefore, by symmetry, we may assume that $(ac)$ is an edge of $G$. Similarly the edge $(ef)$ cannot form a gap with $(ab)$ and $(cd)$. By symmetry, we may assume that $G$  must have one of the following edges: $(ea)$, $(eb)$,$(ec)$, $(ed)$. According to these four  different cases we provide a path  in $G_{I^{[2]}}^{(u,v)}$ connecting $u$ and $v$:
		
		\begin{enumerate}
			\item[(i)] $u, (ea)(cd), (ac)(ef), v$,  if $(ea)\in E(G)$;

			\item[(ii)] $u, (eb)(ac), v$,  if  $(eb)\in E(G)$;

			\item[(iii)] $u, (ec)(ab), v$,  if  $(ec)\in E(G)$;

			\item[(iv)]  $u, (ed)(ab), v$,  if  $(ed)\in E(G)$.
		\end{enumerate}
		
		\medskip
		Case~2. $v=(ac)(ef)$.  Since $\nu_0(G)\leq 2$, by symmetry we may assume $(ec)\in E(G)$ or $(ed)\in E(G)$. In the first case we have the path $u,(ec)(ab), v$, and in the second case we have the path $u,(ed)(ac), v$ in $G_{I^{[2]}}^{(u,v)}$.

		\medskip
		Case~3.  $v=(ae)(bf)$. Since $\nu_0(G)\leq 2$, by symmetry we may assume $(ec)\in E(G)$ or $(ac)\in E(G)$. In the first case we have the path $u,(ec)(ab), v$, and in the second case we have the path $u,(bf)(ac), v$ in $G_{I^{[2]}}^{(u,v)}$.
		
		\medskip
		Case~4. $v=(ae)(cf)$.  We have the path $u,(ab)(cf), v$ in $G_{I^{[2]}}^{(u,v)}$.
	\end{proof}

	\section{Squarefree powers with linear resolutions}
	
In this section we study squarefree powers of edge ideals with linear resolution. As a main tool we use the following

\begin{Lemma}
	\label{general}
	Let $G$ be a graph and let $e=(ab)$ be an edge of $G$ satisfying the following conditions:
	\begin{enumerate}
		\item $I(G-e)^{[k]}$ has linear resolution;
		\item $I(G-\{a,b\})^{[k-1]}$ has linear resolution;
		\item $L(G,e,k):=I(G-e)^{[k]}\cap (ab)I(G-\{a,b\})^{[k-1]}$ has  $(2k+1)$-linear resolution.
	\end{enumerate}
	Then $I(G)^{[k]}$ has linear resolution.
\end{Lemma}

\begin{proof}
	Note that $I(G)^{[k]}=I(G-e)^{[k]}+ (ab)I(G-\{a,b\})^{[k-1]}$. So, we have the following short exact sequence
	\[
	0\to L(G,e,k)\to I(G-e)^{[k]}\dirsum (ab)I(G-\{a,b\})^{[k-1]}\to I(G)^{[k]}\to 0.
	\]
	Thus, we obtain the long exact sequence
	\begin{eqnarray}
	\nonumber
	\cdots &\to& \Tor_i(L(G,e,k),K)\to \Tor_i(I(G-e)^{[k]},K)\dirsum \Tor_i((ab)I(G-\{a,b\})^{[k-1]},K)\\
	\nonumber
	&\to& \Tor_i(I(G)^{[k]},K)\to \cdots
	\end{eqnarray}
	from which the desired conclusion follows by considering its graded components.
\end{proof}

In view of Lemma~\ref{general}, it is important to have a description of the generators of the ideal $L(G,e,k)$.

\begin{Lemma}\label{generators degree}
	Let $e=(ab)$ be an edge of $G$. Suppose that for any $e_1\cdots e_k\in I(G-e)^{[k]}$ with $e_i\in E(G-e)$ for $i=1,\dots ,k$, there exists a vertex $c$ with
	$c\neq a,b$ such that $(ca)$ or $(cb)\in E(G)$, and there exists some $j$ with $c|e_j$ and $e_1\cdots \hat{e_j}\cdots e_k\in I(G-\{a,b\})^{[k-1]}$. Then $L(G,e,k)$ is generated in degree $2k+1$.
\end{Lemma}

\begin{proof}
	Let $u\in G(L(G,e,k))$. Then
	\[
	u=\lcm (e_1\cdots e_k, (ab)e'_1\cdots e'_{k-1})
	\]	
	where $e_1\cdots e_k\in I(G-e)^{[k]}$ with $e_i\in E(G-e)$ for all $i\in[k]$ and $e'_1\cdots e'_{k-1}\in I(G-{\{a,b\}})^{[k-1]}$ with $e'_{\ell}\in E(G-\{a,b\})$ for all $\ell\in[k-1]$. For $e_1\cdots e_k$, there exist some $c$ and $j$ which satisfy the conditions of the lemma, and we may assume that $(cb)\in E(G)$. Note that the squarefree monomial  $v=c(ab)e_1\cdots \hat{e_j}\cdots e_k$ divides $u$.  Moreover, observe that $v\in L(G,e,k)$. Indeed, since
	$e_1\cdots \hat{e_j}\cdots e_k\in I(G-\{a,b\})^{[k-1]}$, it follows that $v\in (ab)I(G-\{a,b\})^{[k-1]}$. On the other hand, $v\in I(G-e)^{[k]}$, because $(cb)e_1\cdots \hat{e_j}\cdots e_k\in I(G-e)^{[k]}$. Since $v$ is of degree $2k+1$, the desired conclusion follows.
\end{proof}

It follows from the proof of Lemma~\ref{generators degree} that under the assumptions of the lemma we have the following description of the generators of $L(G,e,k)$.

\begin{Corollary}\label{generators}
	Let $e=(ab)$ be an edge of $G$. Suppose the conditions of Lemma~\ref{generators degree} are satisfied. Then $L(G,e,k)$ is generated by the monomials of the form
	\[
	c(ab)e_1\cdots e_{k-1}
	\]
	where {\em(i)} $c\neq a,b$,
	{\em(ii)} $(ca)$ or $(cb)\in E(G)$, and
	{\em(iii)}~$e_1\cdots e_{k-1}\in I(G-\{a,b,c\})^{[k-1]}$.
\end{Corollary}

\begin{Remark}\label{rk:perfect matching}
	Let $G$ be a tree on $n$ vertices. It was shown in \cite[page 1089]{BHZ} that $\nu_0(G)\geq \nu(G)-1$. In particular, this implies that $\nu_0(G)= \nu(G)-1$ when $n>2$ and $G$ admits a perfect matching.
\end{Remark}
As an application of Lemma~\ref{general} and Corollary~\ref{generators}, we obtain  

\begin{Theorem}\label{sara}
	Let $G$ be a tree which admits a perfect matching. Then $I(G)^{[\nu_0(G)]}$ has linear resolution. 	
\end{Theorem}

\begin{proof}
	Suppose that $G$ has $2t$ vertices. If $G$ consists only of one edge, then $\nu_0(G)=\nu(G)$ and the result follows from Theorem~\ref{highest power}. Thus assume that $G$ has more than one edge. Then we have $\nu_0(G)=\nu(G)-1=t-1$ by Remark~\ref{rk:perfect matching}. Now we use induction on $t\geq 2$. If $t=2$, then $G$ is just a path on 4 vertices. Then $\nu_0(G)=1$, and hence the result follows, since  $I(G)$ has linear resolution. We assume that $t\geq 3$. By
	\cite[Lemma~2.4]{HL}, there exists a leaf $a$, i.e. a vertex of degree~1, which has a neighbor $b$ of degree 2. We denote the edge $(ab)$ by $e$. Then we have
	\[
	I(G)^{[t-1]}=I(G-e)^{[t-1]}+(ab)I(G-\{a,b\})^{[t-2]}.
	\]
	
	Note that $I(G-e)^{[t-1]}$ has linear resolution by Theorem~\ref{highest power}, since $\nu(G-e)=t-1$. Indeed, to see that $\nu(G-e)=t-1$, let $M$ be a perfect matching for $G$. Then $e\in M$, since $a$ is a leaf. Therefore, $M\setminus \{e\}$ is a matching for $G-e$, and hence $\nu(G-e)\geq t-1$. If $\nu(G-e)=t$, then we get a contradiction, since $|V(G-e)|=2t$ and the vertex $a$ is an isolated vertex of $G-e$. 
	
	Next, observe that $G-\{a,b\}$ is a tree which admits a perfect matching with $\nu(G-\{a,b\})=t-1$. Therefore, $\nu_0(G-\{a,b\})=t-2$ by Remark~\ref{rk:perfect matching}. Then by our induction hypothesis $I(G-\{a,b\})^{[t-2]}$ has linear resolution.
	
	Now we show that $L(G,e,t-1)$ is generated in degree $2t-1$. For this purpose, we observe that the conditions of Lemma~\ref{generators degree} are satisfied for $G$ and $e=(ab)$. Given any $e_1\cdots e_{t-1}\in I(G-e)^{[t-1]}$, with $e_i\in E(G-e)$, we choose $c$ to be the neighbor of $b$ different from $a$. If none of the $e_j$'s contains $c$, then none of them contains $b$ as well. Therefore, $e_1\cdots e_{t-1}$ is of degree at most $2t-3$, a contradiction. Thus, there exists an $e_j$ which contains $c$. Moreover, $e_1\cdots \hat{e_j} \cdots e_{t-1}\in I(G-\{a,b\})^{[t-2]}$, since $e_j$ is the only edge among $e_1,\ldots, e_{t-1}$ which can contain $b$.
	
	It follows from Corollary~\ref{generators} that $L(G,e,t-1)=abc I(G-\{a,b,c\})^{[t-2]}$. Then by using Lemma~\ref{general}, the desired result of the theorem   follows, once we have shown that $I(G-\{a,b,c\})^{[t-2]}$ has linear resolution. This is indeed the case by Theorem~\ref{highest power}, because $\nu(G-\{a,b,c\})=t-2$. In fact, if there is a $(t-1)$-matching in $G-\{a,b,c\}$, then it can be extended by the edge $(bc)$ to a $t$-matching for $G-e$, a contradiction to the fact that $\nu(G-e)=t-1$.
\end{proof}

In view of Theorem~\ref{sara}, the reader may be interested in a criterion for a tree to have a perfect matching. Let $G$ be a tree on $[n]$.  
Even though the following criterion for a tree to admit a perfect matching is a moderate exercise in the graph theory, for the sake of convenience of the reader, its proof is supplied.

\begin{Proposition}
	\label{perfect}
	A tree $G$ on $[n]$ admits a perfect matching if and only if for each $i \in [n]$, exactly one connected component of $G-\{i\}$ consists of an odd number of vertices.
\end{Proposition}

\begin{proof}
	Suppose that $G$ admits a perfect matching $M = \{e_1, \ldots, e_m\}$, where $n = 2m$.  Let $i \in [n]$ belong to $e_j$ and $e_j = \{ i, i' \}$.  Let $G'$ be a connected component of $G-\{i\}$ 
	If $j' \neq j$ and $e_{j'} \cap V(G') \neq \emptyset$, then $e_{j'} \subseteq V(G')$.  In other words, if $j' \neq j$, then one has either $e_{j'} \subseteq V(G')$ or $e_{j'} \cap V(G') = \emptyset$.  Hence, if $i' \not\in V(G')$, then $|V(G')|$ is even.  On the other hand, if $i' \in V(G')$, then $e_j \cap V(G') = \{ i' \}$ and $|V(G')|$ is odd.  Thus, exactly one connected component of $G- \{i\}$, to which $i'$ belongs, consists of an odd number of vertices, as required.
	
	Conversely, suppose that for each $i \in [n]$, exactly one connected component of $G- \{i\}$ consists of an odd number of vertices.  In particular, $n$ is even.  Let, say, $1 \in [n]$ be a leaf and $\{1, 2\}$ be an edge of $G$.  The tree $H$ consisting of just one vertex $1$ is a connected component of $G- \{2\}$.  It then follows that each connected component $G'$ of $G- \{2\}$ with $G' \neq H$ consists of an even number of vertices.
	
	Now, we claim that $G'$ satisfies the condition that, for each $i \in V(G')$, exactly one connected component of $G'-\{i\}$ consists of an odd number of vertices.  Let $G^*, G_1, \ldots, G_q$ be the connected component of $G_{[n] \setminus \{i\}}$ with $2 \in V(G^*)$. Note that, in particular, if the vertex $i$ is adjacent to the vertex 2 in $G$, then $G_0^{*}$ is just the empty graph, (i.e. with no vertices). Then exactly one of $G^*, G_1, \ldots, G_q$ consists of an odd number of vertices.  Let $G^*_0$ be the connected components of $G^*-\{2\}$ for which $V(G^*_0) \subset V(G')$.  Then the connected components of $G'- \{i\}$ are $G^*_0, G_1, \ldots, G_q$.  A crucial fact is
	\[
	|V(G^*_0)| = |V(G^*)| - (n - |V(G')|).
	\]
	Furthermore, $n - |V(G')|$ is even.  If exactly one of $G_1, \ldots, G_q$ consists of an odd number of vertices, then $|V(G^*)|$ is even.  Hence $|V(G^*_0)|$ is even.  If each of $G_1, \ldots, G_q$ consists of an even number of vertices, then $|V(G^*)|$ is odd.  Hence $|V(G^*_0)|$ is odd.
	
	Thus, by using induction on the number of vertices, it follows that each connected component $G'$ of $G- \{2\}$ with $G' \neq H$ possesses a perfect matching.  Hence, together with $\{1, 2\}$, the tree $G$ possesses a perfect matching, as desired.
\end{proof}

	\section{Classification of forests $G$ with $\nu_0(G)\leq 2$}
	
	We are interested in determining all graphs with the property that $I(G)^{[2]}$ has linear resolution. In this section we give a partial answer to this question.
	
	We call a sequence of monomials $u_1,\ldots, u_m$ of same degree to have  {\em linear quotient order}, if $(u_1,\ldots,u_{j-1}):u_j$ is generated by variables for $j=2,\ldots,m$.
	
	Let $u$ and $v$ be monomials of same degree. We set
	\[
	u:v=\frac{u}{\gcd(u,v)}.
	\]
	
	Let $u_1,\ldots, u_m$ be a sequence of monomials of same degree. Let $1\leq \ell<j\leq m$. We say that $u_{\ell}:u_j$ satisfies property~$(*)$, if there exists $\ell'<j$ such that $u_{\ell'}:u_j$ is of degree~1 and $u_{\ell'}:u_j$ divides $u_{\ell}:u_j$. It is clear that the sequence $u_1,\ldots, u_m$ has linear quotient order if and only if $u_{\ell}:u_j$ satisfies property~$(*)$, for all $\ell$ and $j$ with $1\leq \ell<j\leq m$.
	
	Let $u_1,\ldots, u_m$ be sequence of squarefree monomials with linear quotient order. We fix an index $i$ and  introduce the new variables $x_{i1},\ldots,x_{ir}$. For each $j=1,\ldots ,m$, for which $x_i$ divides $u_j$, we set $u_{jk}=x_{ik}(u_j/x_i)$. In other words, $u_{jk}$ is obtained from $u_j$ by replacing $x_i$ by $x_{ik}$. Now we define a new sequence by replacing in $u_1,\ldots, u_m$  each $u_j$ for which $x_i$ divides $u_j$  by the sequence $u_{j1},\ldots, u_{jr}$. The new sequence obtained is called a {\em proliferation} of the sequence $u_1,\ldots, u_m$.
	
	The following example demonstrates this concept: given the sequence $u_1,u_2,u_3=x_1x_2, x_2x_3, x_1x_3$, we choose $i=1$ and the variables $x_{11},x_{12}$. Then the proliferation of $u_1,u_2,u_3$ with respect to these choices  is the sequence
	\[
	u_{11},u_{12},u_2, u_{31},u_{32}=x_{11}x_2, x_{12}x_2, x_2x_3, x_{11}x_3, x_{12}x_3.
	\]
	
	For the proof of the main result of this section the following lemma is useful.
	
	\begin{Lemma}\label{proliferation}
		Let  $u_1,\ldots, u_m$ be a sequence of squarefree monomials with linear quotient order. Then any proliferation of this sequence is again a sequence with linear quotient order.
	\end{Lemma}
	
	\begin{proof}
		Let $\mathcal{S}=u_1,\ldots, u_m$, and let $\mathcal{S'}$ be a  proliferation sequence of $\mathcal{S}$ with respect to the index $i$. Let $u$ and $v$ be two elements of $\mathcal{S'}$ such that $u$ precedes $v$ in the order of the sequence. We distinguish several cases:
		
		Let $u=u_{\ell}$ and $v=u_{j}$. Then $\ell<j$ and $u:v$ satisfies property~$(*)$ in the subsequence of $\mathcal{S'}$ which is obtained from $\mathcal{S}$ by replacing $u_s$ by $u_{s1}$ whenever $x_i$ divides $u_s$.
		
		Let $u=u_{\ell k}$ and $v=u_{j}$, or $u=u_{\ell}$ and $v=u_{jk}$, or $u=u_{\ell k}$ and $v=u_{jk}$. Then $u:v$ satisfies property~$(*)$ in the subsequence of $\mathcal{S'}$ which is obtained from $\mathcal{S}$ by replacing $u_s$ by $u_{sk}$ whenever $x_i$ divides $u_s$.
		
		Let $u=u_{\ell k_1}$ and $v=u_{jk_2}$. If $\ell=j$, then $u:v$ is just the variable ${x_i}_{k_1}$. So suppose that $\ell<j$. Then $u:v$ is divisible by
		$u_{\ell k_2}:v$ which was discussed above that it satisfies property~$(*)$. Hence $u:v$ satisfies property~$(*)$ as well.
	\end{proof}
	
	Suppose that $a$ is a leaf of a graph $G$, and let  $e=(a,b)\in E(G)$. We construct a new graph $\tilde{G}$ from $G$ by adding the new edges   $e_1=(a_1,b),\ldots, e_t=(a_t,b)$, where the $a_i$'s are all leaves. We say that the resulting graph is obtained by {\em proliferating} the leaf $a$. As an immediate consequence of Lemma~\ref{proliferation}, we get the following
	
	\begin{Corollary}\label{whiskers}
		Let $G$ be a graph obtained by proliferating some of the leaves of a graph $H$. If
		$I(H)^{[k]}$ has linear quotients, then $I(G)^{[k]}$ has linear quotients. 	
	\end{Corollary}
	
	Now we consider three types of forests which we call the graphs of types $G_1$, $G_2$ and $G_3$. These types of graphs are important for us concerning our following classification of forests with $\nu_0(G)\leq 2$.
	
	Let $A=\{a_1,\ldots,a_t\}$, $B=\{b_1,\ldots,b_s\}$ and $C=\{c_1,\ldots,c_r\}$. Then a graph of type $G_1$ is a graph with $V(G_1)=A\cup B \cup C\cup \{x_1,x_2,x_3\}$ and the edges
	\[
	(x_1x_2),(x_2x_3),(a_1x_1),\ldots,(a_tx_1),
	(b_1x_2),\ldots,(b_sx_2),(c_1x_3),\ldots,(c_rx_3).
	\]
	Each of $A$, $B$ and $C$ can be empty. The graph shown in Figure~\ref{G1} is an example of a graph of type $G_1$ with $t=3$, $s=2$ and $r=4$. 
	
	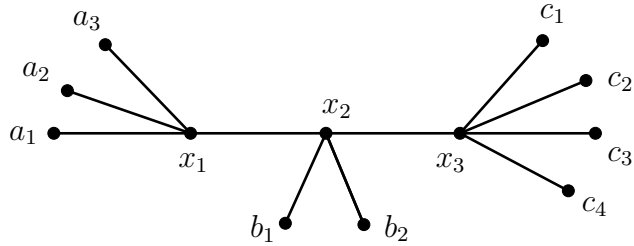
\begin{figure} [h!]
		\centering
		\begin{tikzpicture}[scale = 0.9]
		\draw [line width=1pt] (10.98,-1.06)-- (13,-1.06);
		\draw [line width=1pt] (13,-1.06)-- (15,-1.06);
		\draw [line width=1pt] (15,-1.06)-- (16.98,-1.06);
		\draw (12.66,-1.2) node[anchor=north west] {$x_1$};
		\draw (14.78,-0.4) node[anchor=north west] {$x_2$};
		\draw (16.46,-1.2) node[anchor=north west] {$x_3$};
		\draw (10.16,-0.79) node[anchor=north west] {$a_1$};
		\draw [line width=1pt] (16.98,-1.06)-- (18.98,-1.06);
		\draw (18,1) node[anchor=north west] {$c_1$};
		\draw [line width=1pt] (15,-1.06)-- (15.56,-2.41);
		\draw (15.7,-2.1) node[anchor=north west] {$b_2$};
		\draw [line width=1pt] (13,-1.06)-- (11.18,-0.43);
		\draw [line width=1pt] (13,-1.06)-- (11.74,0.25);
		\draw [line width=1pt] (15.56,-2.41)-- (15,-1.06);
		\draw [line width=1pt] (15,-1.06)-- (14.4,-2.39);
		\draw [line width=1pt] (16.98,-1.06)-- (18.84,-0.28);
		\draw [line width=1pt] (16.98,-1.06)-- (18.2,0.31);
		\draw (13.72,-2.11) node[anchor=north west] {$b_1$};
		\draw (10.35,0.15) node[anchor=north west] {$a_2$};
		\draw (11.1,0.9) node[anchor=north west] {$a_3$};
		\draw (19,0) node[anchor=north west] { $c_2$};
		\draw (19,-1.1) node[anchor=north west] {$c_3$};
		\draw [line width=1pt] (16.98,-1.06)-- (18.58,-1.91);
		\draw (18.62,-1.85) node[anchor=north west] {$c_4$};
		\begin{scriptsize}
		\draw [fill=black] (10.98,-1.06) circle (2.5pt);
		\draw [fill=black] (13,-1.06) circle (2.5pt);
		\draw [fill=black] (15,-1.06) circle (2.5pt);
		\draw [fill=black] (16.98,-1.06) circle (2.5pt);
		\draw [fill=black] (18.98,-1.06) circle (2.5pt);
		\draw [fill=black] (15.56,-2.41) circle (2.5pt);
		\draw [fill=black] (11.18,-0.43) circle (2.5pt);
		\draw [fill=black] (11.74,0.25) circle (2.5pt);
		\draw [fill=black] (14.4,-2.39) circle (2.5pt);
		\draw [fill=black] (18.84,-0.28) circle (2.5pt);
		\draw [fill=black] (18.2,0.31) circle (2.5pt);
		\draw [fill=black] (18.58,-1.91) circle (2.5pt);
		\end{scriptsize}
		\end{tikzpicture}
		\caption{A graph of type $G_1$}
		\label{G1}
	\end{figure}
	Let $A=\{a_1,\ldots,a_t\}$ and $B=\{b_1,\ldots,b_s\}$ be two non-empty sets of vertices. Then a graph of type $G_2$ is a graph with $V(G)=A\cup B\cup \{x_1,x_2,x_3,x_4\}$ and the edges
	\[
	(x_1x_2),(x_2x_3),(x_3x_4),(a_1x_1),\ldots,(a_tx_1),
	(b_1x_4),\ldots,(b_sx_4).
	\]
	The graph shown in Figure~\ref{G2} is an example of a graph of type $G_2$ with $t=4$ and $s=2$. 
	
	\begin{figure}[h!]
		\centering
		\begin{tikzpicture}[scale=0.9]
		\draw [line width=1pt] (10.88,-1.04)-- (12.9,-1.04);
		\draw [line width=1pt] (12.9,-1.04)-- (14.9,-1.04);
		\draw [line width=1pt] (14.9,-1.04)-- (16.82,-1.03);
		\draw (12.56,-1.2) node[anchor=north west] {$x_1$};
		\draw (14.62,-1.2) node[anchor=north west] {$x_2$};
		\draw (16.68,-1.2) node[anchor=north west] {$x_3$};
		\draw (10.06,-0.77) node[anchor=north west] {$a_1$};
		\draw [line width=1pt] (16.82,-1.03)-- (18.88,-1.04);
		\draw (20.64,0.53) node[anchor=north west] {$b_2$};
		\draw [line width=1pt] (12.9,-1.04)-- (11.08,-0.41);
		\draw [line width=1pt] (12.9,-1.04)-- (11.64,0.27);
		\draw (21,-0.73) node[anchor=north west] {$b_1$};
		\draw (10.28,0.17) node[anchor=north west] {$a_2$};
		\draw (11,0.95) node[anchor=north west] {$a_3$};
		\draw (11.78,1.4) node[anchor=north west] {$a_4$};
		\draw [line width=1pt] (18.88,-1.04)-- (20.74,-1.03);
		\draw (18.58,-1.2) node[anchor=north west] {$x_4$};
		\draw [line width=1pt] (18.88,-1.04)-- (20.46,-0.11);
		\draw [line width=1pt] (12.9,-1.04)-- (12.4,0.65);
		\begin{scriptsize}
		\draw [fill=black] (10.88,-1.04) circle (2.5pt);
		\draw [fill=black] (12.9,-1.04) circle (2.5pt);
		\draw [fill=black] (14.9,-1.04) circle (2.5pt);
		\draw [fill=black] (16.82,-1.03) circle (2.5pt);
		\draw [fill=black] (18.88,-1.04) circle (2.5pt);
		\draw [fill=black] (11.08,-0.41) circle (2.5pt);
		\draw [fill=black] (11.64,0.27) circle (2.5pt);
		\draw [fill=black] (20.74,-1.03) circle (2.5pt);
		\draw [fill=black] (20.46,-0.11) circle (2.5pt);
		\draw [fill=black] (12.4,0.65) circle (2.5pt);
		\end{scriptsize}
		\end{tikzpicture}
		\caption{A graph of type $G_2$}
		\label{G2}
	\end{figure}
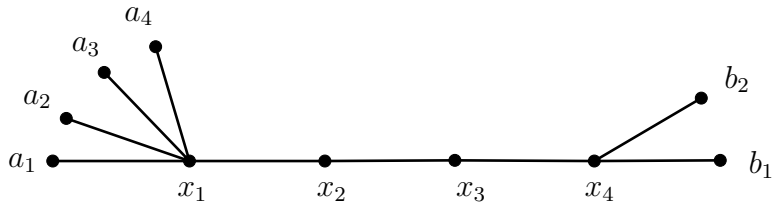
	
	A graph of type $G_3$ is just the disjoint union of two star graphs. Note that any edge is also a star graph.
	
	\medskip
	In the following theorem, $P_n$ denotes the path graph with $n$ vertices, and hence with length~$n-1$. 
	
	\begin{Theorem}\label{tree characterization}
		Let $G\neq P_2$ be a forest with no isolated vertices. Then the following conditions  are equivalent:
		\begin{enumerate}
			\item[(i)] $I(G)^{[2]}$ has linear quotients.
			\item[(ii)] $I(G)^{[2]}$ has linear resolution.
			\item[(iii)] $I(G)^{[2]}$ is linearly related.
			\item[(iv)] $\nu_0(G)\leq 2$.
			\item[(v)] $G$ is of type $G_1$, $G_2$ or $G_3$.
		\end{enumerate}
	\end{Theorem}
	
	\begin{proof}
		(i)\implies (ii)\implies (iii) are trivial, and
		(iii) \implies (iv) follows from \cite[Lemma~5.2]{BHZ}.
		
		(iv) \implies (v): Note that $\nu_0(P_7)=3$, which implies that $G$ is $P_7$-free. Therefore the length of a longest path in $G$ is at most~5.
		
		First we assume that $G$ is connected. If the longest path in $G$ has length~2, then $G$ is of type of $G_1$ with $A=C=\emptyset$. If the length of a longest path in $G$ is~3, then $G$ is of type $G_1$ with $A\neq \emptyset$ and $C=\emptyset$. Suppose the length of a longest path in $G$ is~4. Then $G$ is of type $G_1$ with $A,C\neq \emptyset$. Note that if $B\neq \emptyset$, then all of its elements are leaves of $G$, since otherwise $\nu_0(G)\geq 3$, a contradiction. Suppose a longest path in $G$ is of length~5. Assume that $P:x_1,x_2,x_3,x_4,x_5,x_6$ is a path of $G$. Then the vertices $x_3$ and $x_4$ are of degree~2, since otherwise $\nu_0(G)\geq 3$, a contradiction. Therefore, $G$ is of type $G_2$.
		
		Next, we assume that $G$ is disconnected. Since $\nu_0(G)\leq 2$ and there is no isolated vertices in $G$, it follows that $G$ has exactly two components. On the other hand, each of the connected components has the matching number equal to~1, since otherwise $\nu_0(G)\geq 3$, a contradiction. Therefore, it follows that $G$ has to be a disjoint union of two star graphs, and hence $G$ is of type $G_3$.
		
		(v) \implies (i): We may assume that $\nu(G)>1$, because otherwise $I(G)^{[2]}=(0)$.  First note that, depending on whether the sets $A$, $B$ and $C$ are empty or not, the graphs of type $G_1$, by symmetry, are obtained by proliferating some leaves in one of  the graphs displayed in Figure~\ref{linear quotients}.
		
		\usetikzlibrary{arrows}
		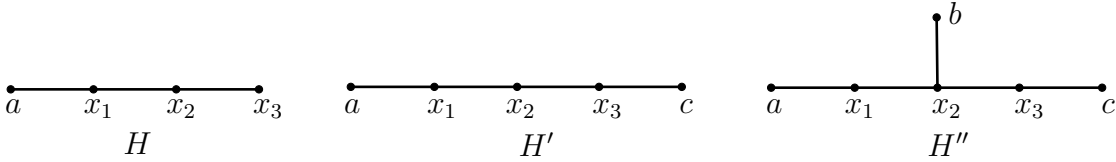
\begin{figure}[h!]
			\begin{tikzpicture}[scale = 0.55]
			\draw [line width=1pt] (3.6,0.98)-- (5.6,0.98);
			\draw [line width=1pt] (5.6,0.98)-- (7.6,0.98);
			\draw [line width=1pt] (7.6,0.98)-- (9.6,0.98);
			\draw (5.1,1) node[anchor=north west] {$x_1$};
			\draw (7.1,1) node[anchor=north west] {$x_2$};
			\draw (3.2,1) node[anchor=north west] {$a$};
			\draw [line width=1pt] (11.82,1.04)-- (13.84,1.04);
			\draw [line width=1pt] (13.84,1.04)-- (15.84,1.04);
			\draw [line width=1pt] (15.84,1.04)-- (17.82,1.04);
			\draw (13.4,1) node[anchor=north west] {$x_1$};
			\draw (15.3,1) node[anchor=north west] {$x_2$};
			\draw (17.4,1) node[anchor=north west] {$x_3$};
			\draw (11.4,1) node[anchor=north west] {$a$};
			\draw [line width=1pt] (17.82,1.04)-- (19.82,1.04);
			\draw (19.5,1) node[anchor=north west] {$c$};
			\draw [line width=1pt] (21.98,1.02)-- (24,1.02);
			\draw [line width=1pt] (24,1.02)-- (26,1.02);
			\draw [line width=1pt] (26,1.02)-- (27.98,1.02);
			\draw (23.5,1) node[anchor=north west] {$x_1$};
			\draw (25.6,1) node[anchor=north west] {$x_2$};
			\draw (27.6,1) node[anchor=north west] {$x_3$};
			\draw (21.6,1) node[anchor=north west] {$a$};
			\draw [line width=1pt] (27.98,1.02)-- (29.98,1.02);
			\draw (29.7,1) node[anchor=north west] {$c$};
			\draw (9.2,1) node[anchor=north west] {$x_3$};
			\draw (6.06,0.2) node[anchor=north west] {$H$};
			\draw (15.62,0.2) node[anchor=north west] {$H'$};
			\draw (25.5,0.2) node[anchor=north west] {$H''$};
			\draw [line width=1pt] (26,1.02)-- (25.98,2.72);
			\draw (26,3.36) node[anchor=north west] {$b$};
			\begin{scriptsize}
			\draw [fill=black] (3.6,0.98) circle (2.5pt);
			\draw [fill=black] (5.6,0.98) circle (2.5pt);
			\draw [fill=black] (7.6,0.98) circle (2.5pt);
			\draw [fill=black] (9.6,0.98) circle (2.5pt);
			\draw [fill=black] (11.82,1.04) circle (2.5pt);
			\draw [fill=black] (13.84,1.04) circle (2.5pt);
			\draw [fill=black] (15.84,1.04) circle (2.5pt);
			\draw [fill=black] (17.82,1.04) circle (2.5pt);
			\draw [fill=black] (19.82,1.04) circle (2.5pt);
			\draw [fill=black] (21.98,1.02) circle (2.5pt);
			\draw [fill=black] (24,1.02) circle (2.5pt);
			\draw [fill=black] (26,1.02) circle (2.5pt);
			\draw [fill=black] (27.98,1.02) circle (2.5pt);
			\draw [fill=black] (29.98,1.02) circle (2.5pt);
			\draw [fill=black] (25.98,2.72) circle (2.5pt);
			\end{scriptsize}
			\end{tikzpicture}
			\caption{The graphs $H$, $H'$ and $H''$}
			\label{linear quotients}
		\end{figure}
		
		The matching number of $H$ and $H'$ is equal to 2, and hence by Theorem~\ref{highest power}, $I(H)^{[2]}$ and $I(H')^{[2]}$ have linear quotients. The ideal $I(H'')^{[2]}$ has also linear quotients given by the following ordering of its generators:
		\[
		(ax_1)(x_2b), (ax_1)(x_2x_3), (ax_1)(x_3c), (x_1x_2)(x_3c),(x_2b)(x_3c).
		\]
		
		On the other hand, the graphs of type $G_2$ are obtained by proliferating some leaves in the graph $\tilde{H}$ depicted in Figure~\ref{linear quotients2}. 
		
		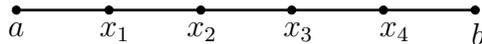
\begin{figure}[h!]
			\centering 
			\begin{tikzpicture}[scale = 0.61]
			\draw [line width=1pt] (11.82,1.04)-- (13.84,1.04);
			\draw [line width=1pt] (13.84,1.04)-- (15.84,1.04);
			\draw [line width=1pt] (15.84,1.04)-- (17.82,1.04);
			\draw [line width=1pt] (19.82,1.04)-- (21.82,1.04);
			\draw (13.4,1) node[anchor=north west] {$x_1$};
			\draw (15.3,1) node[anchor=north west] {$x_2$};
			\draw (17.4,1) node[anchor=north west] {$x_3$};
			\draw (11.4,1) node[anchor=north west] {$a$};
			\draw [line width=1pt] (17.82,1.04)-- (19.82,1.04);
			\draw (19.5,1) node[anchor=north west] {$x_4$};
			\draw (21.5,1) node[anchor=north west] {$b$};
			\begin{scriptsize}
			\draw [fill=black] (11.82,1.04) circle (2.5pt);
			\draw [fill=black] (13.84,1.04) circle (2.5pt);
			\draw [fill=black] (15.84,1.04) circle (2.5pt);
			\draw [fill=black] (17.82,1.04) circle (2.5pt);
			\draw [fill=black] (19.82,1.04) circle (2.5pt);
			\draw [fill=black] (21.82,1.04) circle (2.5pt);
			\end{scriptsize}
			\end{tikzpicture}
			\caption{The graph $\tilde{H}$}
			\label{linear quotients2}
		\end{figure}

		The second squarefree power of the edge ideal of $\tilde{H}$ has linear quotients given by the following order of its generators:
		\[
		(ax_1)(x_2x_3), (ax_1)(x_3x_4), (ax_1)(x_4b), (x_1x_2)(x_3x_4),(x_1x_2)(x_4b), (x_2x_3)(x_4b).
		\]
		The graphs of type $G_3$ are obtained by proliferating some leaves in a disjoint union of two edges. In this case the second squarefree power of the edge ideal is a principal ideal.
		
		Therefore, Corollary~\ref{whiskers} implies that $I(G)^{[2]}$ has linear quotients, as desired.
	\end{proof}

	\section{The squarefree Ratliff property}

	Ratliff in his paper \cite{R}  showed that for any ideal $I$ in a Noetherian ring, $I^{k+1} : I = I^k$ for $k\gg 0$, and that $I^{k+1} : I = I^k$
	for all $k$ if $I$ is a normal ideal.  Independent of normality, Mart\'{i}nez-Bernal, Morey and Villarreal \cite{MMV} showed that for
	any edge ideal $I$ one has $I^{k+1} : I = I^k$ for all $k$. An ideal satisfying this equality for all $k$ is  said to satisfy the {\em Ratliff property}. Besides of edge ideals, any ideal whose all powers have linear resolution satisfy the Ratliff property, as shown in \cite{HQ}. For squarefree powers of monomial ideals the behaviour is quite different. We say that a squarefree monomial ideal satisfies the {\em squarefree Ratliff property}, if $I^{[k]}:I=I^{[k]}$ for all $k\geq 2$.
	
	Surprisingly one has
	
	\begin{Theorem}\label{Surprised}
		Any nonzero squarefree monomial ideal satisfies the squarefree Ratliff property. 
	\end{Theorem}
	
	\begin{proof}
		Let $I$ be a nonzero squarefree monomial ideal. First note that for all $k$, we have $I^{[k]}\subseteq I^{[k]}:I$. Conversely, let $u\in G(I^{[k]}:I)$ and assume that $u\notin I^{[k]}$. Since $I^{[k]}:I$ is a squarefree monomial ideal, $u$ is squarefree. 
		
		First assume that $u\in I$. Then there exists some $v\in G(I)$ and a monomial $w\in S$ such that $u=vw$. Since $u\in G(I^{[k]}:I)$, it follows that $vu=wv^2\in I^{[k]}$. Since $I^{[k]}$ is a squarefree monomial ideal, we  deduce that $u=wv\in I^{[k]}$, a contradiction. 
		
		Next, suppose that $u\notin I$. Let $v\in G(I)$. Since $u\in G(I^{[k]}:I)$, it follows that $uv\in I^{[k]}$. Suppose that $v=v'v''$ where $\gcd(v',u)=1$ and $v''$ divides $u$. We have $uv'\in I^{[k]}$, since $I^{[k]}$ is a squarefree monomial ideal. 
		Therefore, there exist $w_1,\ldots, w_k\in G(I)$ and a monomial $w$ such that $uv'=w_1\cdots w_k w$,  where $\supp (w_i)$'s are pairwise disjoint. Since $u\in I^{[k]}:I$, we have $uw_i\in I^{[k]}$ for all $i=1,\ldots,k$. Then, by induction on the cardinality of the set
		$\supp (v)\setminus \supp(u)$, we get the contradiction to the assumption $u\notin I^{[k]}$, and hence the result follows.    	
	\end{proof}
	
	Motivated by Theorem~\ref{Surprised}, one may ask whether we also have $I^{[k]}:I^{[\ell]}=I^{[k]}$, for any squarefree monomial ideal $I$ and integers $1< \ell<k$. Observe that if $I^{[k]}:I^{[\ell]}=I^{[k]}$, then $I^{[k]}:I^{[\ell-1]}=I^{[k]}$. Indeed,   
	\[
	I^{[k]} \subseteq I^{[k]} : I^{[\ell-1]} \subseteq I^{[k]} : I^{[\ell]} = I^{[k]}.
	\]
	
	Regarding the above question, we have partial answers when $I$ is the edge ideal of a graph. 
	
	\begin{Theorem}
		\label{easyratliff}
		Let $G$ be a graph with  no isolated vertex.  Let $k$ be an integer with $2 < k \leq \nu(G)$.  Then
		\[
		I(G)^{[k]} : I(G)^{[2]} = I(G)^{[k]}.
		\]
	\end{Theorem}
	
	\begin{proof}
		Let $I = I(G)$ and $V(G)=[n]$. Note that we have 
		$I^{[k]}\subseteq I^{[k]} : I^{[2]}$. Now, we show $I^{[k]} : I^{[2]} \subseteq I^{[k]}$.  Let $u \in I^{[k]} : I^{[2]}$ be a squarefree monomial, and let $V \subseteq [n]$ be the support of $u$.  If $\nu(G_V) \geq 2$, then there is a squarefree monomial $w \in I^{[2]}$ which divides $u$.  Since $w u \in I^{[k]}$ and since $u = \sqrt{w  u}$, one has $u \in I^{[k]}$. Next, we show that $\nu(G_V)\geq 2$. 
		
		Suppose that $\nu(G_V) = 0$.  Thus $G_V$ possesses no edge of $G$.  Let $i \in V$.  Since $G$ has no isolated vertex, there is $e=\{i,\ell\} \in E(G)$.  Then $\nu(G_{V \cup \{\ell\}}) = 1$.  Since $\nu(G) \geq 3$, there is a $3$-matching $\{ f_1, f_2, f_3 \}$ of $G$.  Let, say, $e \cap f_1 = \emptyset$.  If $f_1 \cap (V \setminus e) = \emptyset$ and each of $j \in V \setminus e$ satisfies $N_G(j) \cap f_1 = \emptyset$, then set $f = f_1$.  If $f_1 \cap (V \setminus e) = \emptyset$ and if there is $j \in V \setminus e$ with $N_G(j) \cap f_1 \neq \emptyset$, then set $f = \{j\} \cup (N_G(j) \cap f_1)$.  If there is $j \in V \setminus e$ with $j \in f_1$, then set $f = f_1$.  It then follows that $\nu(G_{V \cup e \cup f}) = 2$.  Let $w$ be the squarefree monomial whose support is $e \cup f$.  Then $w \in I^{[2]}$.  Since the support of $u  w$ coincides with $V \cup e \cup f$ and since $k > 2$, it follows that $uw \not\in I^{[k]}$, and hence $u \not\in I^{[k]} : I^{[2]}$, a contradiction. Therefore, $\nu(G_V)\neq 0$. 
		
		On the other hand, if $\nu(G_V) = 1$, then the above technique with adding a suitable $f$ is valid and we deduce that $u \not\in I^{[k]} : I^{[2]}$, a contradiction. Thus, $\nu(G_V)\neq 1$.   
	\end{proof}

	A graph is called {\em equimatchable} if every maximal matching is maximum, (see~\cite[p.~102]{LP}).
	
	\begin{Lemma}
		\label{happy}
		Let $G$ be an equimatchable graph and $V \subseteq [n]$.  If $\nu(G_V) < \nu(G)$, then there is $e \in E(G)$ for which $\nu(G_{V \cup e}) = \nu(G_V) + 1$.
	\end{Lemma}
	
	\begin{proof}
		Let $k = \nu(G_V)$, and let $\{e_1, \ldots, e_k\}$ be a maximum matching of $G_V$.  Since $G$ is equimatchable and $k < \nu(G)$, there is $e \in E(G)$ for which $\{e_1, \ldots, e_k, e\}$ is a $(k + 1)$-matching of $G$.  Let $V' = V \setminus (e_1\cup \cdots \cup e_k)$ and $e = \{i, j\}$.  If $N_G(i) \cap V' = \emptyset$ and $N_G(j) \cap V' = \emptyset$, then $\nu(G_{V \cup e}) = \nu(G_V) + 1$.  On the other hand, if, say, $N_G(j) \cap V' \neq \emptyset$ and $e' = \{j, \ell\} \in E(G)$, where $\ell \in V'$, then $\{e_1, \ldots, e_k, e'\}$ is a $(k + 1)$-matching of $G$ and $\nu(G_{V \cup e'}) = \nu(G_{V \cup \{j\}}) = \nu(G_V) + 1$, as desired.
	\end{proof}
	
	\begin{Corollary}
		\label{veryhappy}
		Let $G$ be an equimatchable graph and $V \subseteq [n]$.  If $\nu(G_V) < \nu(G)$, then there is a sequence of edges $e_1, \ldots, e_r$, where $r = \nu(G) - \nu(G_V)$, such that 
		\[
		\nu(G_{V \cup e_1 \cup \cdots \cup e_j}) = \nu(G_V) + j
		\]
		for $1 \leq j \leq r$.
	\end{Corollary}
	
	\begin{Theorem}
		\label{ratliff}
		Suppose that $G$ is an equimatchable graph.  Let $k$ be an integer with $1 < k \leq \nu(G)$.  Then
		\[
		I(G)^{[k]} : I(G)^{[k-1]} = I(G)^{[k]}.
		\]
	\end{Theorem}
	
	\begin{proof}
		Let $I = I(G)$.  One has $I^{[k]} \subseteq I^{[k]} : I^{[k-1]}$.  To see why $I^{[k]} : I^{[k-1]} \subseteq I^{[k]}$ is true, suppose that $I^{[k]} : I^{[k-1]} \not\subseteq I^{[k]}$.  Then there is a squarefree monomial $u \in I^{[k]} : I^{[k-1]}$ with $u \not\in I^{[k]}$.  Let $V \subseteq [n]$ be the support of $u$.  Since $u \not\in I^{[k]}$, it follows that $\nu(G_V) < k$.  Corollary \ref{veryhappy} then guarantees the existence of edges $e_1, \ldots, e_r$, where $r = (k - 1) - \nu(G_V)$, for which
		\[
		\nu(G_{V \cup e_1 \cup \cdots \cup e_r}) = k - 1
		\]
		Let $f_1, \ldots, f_{k-1}$ be a maximum matching of $G_{V \cup e_1 \cup \cdots \cup e_r}$ and $w \in I^{[k-1]}$ the squarefree monomial whose support is $f_1 \cup \cdots \cup f_{k-1}$.  Now, the support of the monomial $uw$ is a subset of $V \cup e_1 \cup \cdots \cup e_r$.  Hence $uw \not\in I^{[k]}$.  This contradicts the fact that $u \in I^{[k]} : I^{[k-1]}$.  Hence $I^{[k]} : I^{[k-1]} \subseteq I^{[k]}$, as required.
	\end{proof}
	
	\begin{Corollary}
		\label{againratliff}
		Suppose that $G$ is an equimatchable graph.  Let $k$ and $\ell$ be integers with $1 \leq \ell < k \leq \nu(G)$.  Then
		\[
		I(G)^{[k]} : I(G)^{[\ell]} = I(G)^{[k]}.
		\]
	\end{Corollary}

	\section*{Acknowledgment}
	
	The present paper was completed while the authors stayed at Mathematisches Forschungsinstitut in Oberwolfach, August~18 to September~7, 2019, in the frame of the Research in Pairs Program.  
	Nursel Erey was supported by T\"UB\.ITAK project no. 118C033.
	Takayuki Hibi was partially supported by JSPS KAKENHI 19H00637.  
	Sara Saeedi Madani was in part supported by a grant from IPM (No. 98130013).

\end{document}